\documentclass{article}

\usepackage[dvipsnames]{xcolor}
\usepackage[margin=2.5cm]{geometry}
\usepackage{microtype}

\usepackage{mathtools}
\usepackage{amssymb}
\usepackage{amsthm}

\usepackage{newpxmath}
\usepackage{newpxtext}

\usepackage{bm}
\usepackage{dsfont}

\usepackage{array,multirow}
\usepackage{thmtools}
\usepackage{thm-restate}
\usepackage{authblk}
\usepackage{setspace}
\usepackage[shortlabels]{enumitem}
\usepackage[linesnumbered,commentsnumbered,ruled,vlined]{algorithm2e}
\usepackage[noadjust]{cite}
\usepackage{booktabs}
\usepackage[colorinlistoftodos, shadow, linecolor=white, backgroundcolor=white]{todonotes}

\usepackage{hyperref}
\newcommand\myshade{100}
\colorlet{mylinkcolor}{NavyBlue}
\colorlet{mycitecolor}{YellowOrange}
\colorlet{myurlcolor}{Aquamarine}
\hypersetup{
  linkcolor  = mylinkcolor!\myshade!black,
  citecolor  = mycitecolor!\myshade!black,
  urlcolor   = myurlcolor!\myshade!black,
  colorlinks = true,
  pdftitle   = {Circuit Diameter of Polyhedra is Strongly Polynomial},
  pdfauthor  = {Bento Natura},
  pdfsubject = {Combinatorial optimization, polyhedral geometry},
  pdfkeywords = {circuit diameter, polyhedra, polynomial Hirsch conjecture, linear programming, strongly polynomial},
}
\usepackage{cleveref}
\newcommand{\LineRef}[1]{Line~\ref{#1}}
\newcommand{\LineRefRange}[2]{Lines~\ref{#1} to~\ref{#2}}

\newcommand{\circuits}{\mathcal{C}}
\newcommand{\EE}{\mathcal E}
\newcommand{\aug}{\operatorname{aug}}

\newcommand{\supp}{\mathrm{supp}}
\newcommand{\0}{\mathbf{0}}
\newcommand{\rk}{\operatorname{rk}}
\newcommand{\set}[1]{\left\{ #1 \right\}}
\newcommand{\refidx}{r}
\newcommand{\eps}{\varepsilon}
\newcommand{\poly}{\operatorname{poly}}
\newcommand{\thresh}{\tau}
\newcommand{\R}{\mathbb{R}}

\DeclareMathOperator*{\argmax}{arg\,max}
\DeclareMathOperator*{\argmin}{arg\,min}

\renewcommand{\phi}{\varphi}
\renewcommand{\rho}{\varrho}

\foreach \x in {A,...,Z}{%
    \expandafter\xdef\csname m\x\endcsname{\noexpand\mathbf{\x}}
}
\foreach \x in {Alpha,...,Omega}{%
   \expandafter\xdef\csname m\x\endcsname{\noexpand\mathbf{\x}}
}

\newtheorem{theorem}{Theorem}[section]
\newtheorem{lemma}[theorem]{Lemma}

\newtheorem{corollary}[theorem]{Corollary}

\newtheorem{claim}[theorem]{Claim}

\theoremstyle{definition}
\newtheorem{definition}[theorem]{Definition}

\newenvironment{claimproof}[1][\proofname]
{
    \proof[#1 of Claim]%
        
}
{
    \endproof
}

\newcounter{Hequation}

\makeatletter
\g@addto@macro\equation{\stepcounter{Hequation}}
\makeatother

\title{Circuit Diameter of Polyhedra is Strongly Polynomial}
\date{}

\author{Bento Natura}
\affil{Columbia University}

\begin{document}

\maketitle

\begin{abstract}
We prove a strongly polynomial bound on the circuit diameter of polyhedra, resolving the circuit analogue of the polynomial Hirsch conjecture. Specifically, we show that the circuit diameter of a polyhedron $P = \{x\in \R^n:\, \mA x = b, \, x \ge \0\}$ with $\mA\in\R^{m\times n}$ is $O(m^2 \log m)$. Our construction yields monotone circuit walks, giving the same bound for the monotone circuit diameter. 

The circuit diameter, introduced by Borgwardt, Finhold, and Hemmecke (SIDMA 2015), is a natural relaxation of the combinatorial diameter that allows steps along circuit directions rather than only along edges. All prior upper bounds on the circuit diameter were only weakly polynomial. Finding a circuit augmentation algorithm that matches this bound would yield a strongly polynomial time algorithm for linear programming, resolving Smale's 9th problem.
\end{abstract}

\section{Introduction}

Linear programming is a cornerstone of modern optimization, with applications spanning operations research, economics, machine learning, and algorithm design. The simplex method, introduced by Dantzig in 1947, remains a workhorse algorithm for solving linear programs in practice, despite having exponential worst-case complexity. Understanding the geometry underlying the simplex method's remarkable empirical performance has been a central pursuit in optimization theory for over seventy years.

At the heart of this question lies the \emph{combinatorial diameter} of a polyhedron: the diameter of its vertex-edge graph, measuring the maximum number of edge steps needed to traverse between any pair of vertices. This quantity directly governs the worst-case iteration complexity of simplex-type algorithms. In 1957, Warren M. Hirsch conjectured that the combinatorial diameter of a $d$-dimensional polytope with $f$ facets is at most $f-d$, a bound that would elegantly explain why the simplex method rarely requires many pivots in practice. Hirsch's conjecture became one of the most celebrated open problems in polyhedral combinatorics, influencing decades of research in discrete geometry and optimization.

The conjecture remained open for over half a century. In a landmark result, Santos \cite{santos2012} disproved it in 2012, constructing a counterexample with diameter exceeding the Hirsch bound. However, Santos's counterexample still has polynomial diameter, leaving the more fundamental question unresolved: the \emph{polynomial Hirsch conjecture}, which asks whether the combinatorial diameter of every polytope admits a polynomial bound $\poly(f,d)$, remains the central open problem in this area. This question is not merely of geometric interest: a positive answer would fundamentally advance our understanding of why pivot algorithms perform well, while a negative answer would reveal inherent limitations of the simplex method and related approaches.

Current progress on the polynomial Hirsch conjecture remains limited. Kalai and Kleitman \cite{KK1992} established the first quasipolynomial bound; subsequent improvements have refined the exponent \cite{Sukegawa2017}, but no polynomial bound is known for general polytopes. Polynomial bounds have been proven only in restricted cases: Dyer and Frieze \cite{Dyer1994} settled the conjecture for totally unimodular matrices, and polynomial bounds in terms of the maximum subdeterminant $\Delta$ have been established for integer constraint matrices \cite{Bonifas2014,brunsch2013,eisenbrand2017,dadush2016shadow}. After 65 years of effort, the polynomial Hirsch conjecture for general polyhedra remains tantalizingly open.

\subsection{Circuit Diameter: A Natural Geometric Relaxation}

Given the difficulty of the combinatorial diameter problem, it is natural to study relaxations that may be more tractable while still capturing essential geometric properties of polyhedra. Borgwardt, Finhold, and Hemmecke \cite{Borgwardt2015} introduced the \emph{circuit diameter} as such a relaxation, generalizing edge directions to \emph{circuit directions}, which encompass all possible edge directions across all choices of the right-hand side vector.

To define circuit diameter precisely, consider a polyhedron in standard equality form
\begin{equation}\label{sys:polytope}\tag{P}
P=\{\,x\in \R^n: \mA x=b, x\ge \0\,\}
\end{equation}
for $\mA\in \R^{m\times n}$, $b\in \R^m$, with $\rk(\mA)= m$. An \emph{elementary vector} in $\ker(\mA)$ is a support-minimal nonzero vector $g\in \ker(\mA)$: no $h\in \ker(\mA)\setminus\{\0\}$ satisfies $\supp(h)\subsetneq \supp(g)$. A \emph{circuit} is the support of an elementary vector; these are precisely the circuits of the linear matroid defined over the columns of $\mA$. We denote by $\EE(\mA)\subseteq \ker(\mA)$ and $\circuits(\mA)\subseteq 2^{[n]}$ the sets of elementary vectors and circuits, respectively. We remark that many papers on circuit diameter \cite{borgwardt2018circuit,Borgwardt2016-hierarchy,Borgwardt2015,Borgwardt2018,Kafer2019} refer to elementary vectors as circuits; we follow the traditional convention of \cite{Fulkerson67,Rockafellar67,Lee89}.

Elementary vectors naturally generalize edge directions. Every edge direction of $P$ is an elementary vector, and conversely, the set $\EE(\mA)$ equals the set of all possible edge directions of polyhedra of the form \eqref{sys:polytope} as $b$ varies over $\R^m$ \cite{Sturmfels1997}. Thus, elementary vectors capture the full geometric richness of the constraint matrix $\mA$, independent of any particular right-hand side.

A \emph{circuit walk} is a sequence $x^{(0)},x^{(1)},\ldots,x^{(k)}$ in $P$ where each step $x^{(i+1)}=x^{(i)}+\alpha^{(i)} g^{(i)}$ uses an elementary vector $g^{(i)}\in \EE(\mA)$ with $\alpha^{(i)} > 0$ chosen maximally: $x^{(i)}+\alpha' g^{(i)}\notin P$ for any $\alpha'>\alpha^{(i)}$. The \emph{circuit diameter} of $P$ is the maximum, over all pairs of vertices $x,y\in P$, of the length of a shortest circuit walk from $x$ to $y$. Note that circuit walks are non-reversible due to the maximal step requirement; this asymmetry is a key technical challenge.

Circuit directions arise naturally in classical combinatorial optimization algorithms. Many fundamental algorithms for network flows and transportation problems, including the network simplex method and minimum-cost flow algorithms, are circuit augmentation algorithms \cite{Bland76,DHL15}. Thus, understanding circuit walks is not merely a geometric abstraction, but directly relevant to algorithmic practice.

The \emph{circuit diameter conjecture}, formulated in \cite{Borgwardt2015}, asserts that the circuit diameter of a $d$-dimensional polyhedron with $f$ facets is at most $f-d$. For $P$ in the form \eqref{sys:polytope}, where $d=n-m$ and the number of facets is at most $n$, this conjectured bound is $m$. While this strong form of the conjecture remains open, our focus is on the weaker but more fundamental \emph{polynomial circuit diameter conjecture}: whether circuit diameter is bounded by $\poly(m,n)$.

\paragraph{The circuit diameter conjecture as a central open problem.}

Since its introduction, the circuit diameter conjecture has attracted significant attention in the polyhedral combinatorics and optimization communities, appearing prominently throughout the literature \cite{borgwardt2018circuit,black2025counterexamples,Dadush2024,BGKLS2025,black2025shortcircuitwalksfixed,Kafer2019,NS2024,Wulf2025,SY2015,BFH2016network,BV2022,Ekbatani2021}. Black, Borgwardt, and Brugger~\cite{black2025counterexamples} describe it as ``the main open problem in this area''; Borgwardt, Grewe, Kafer, Lee, and Sanit\`{a}~\cite{BGKLS2025} note that ``a resolution of the circuit diameter conjecture would give insight to the reason the Hirsch conjecture does not hold''.

\subsection{Our Results: The First Strongly Polynomial Bound}

The main contribution of this paper is the first strongly polynomial bound on the circuit diameter of polyhedra, settling the polynomial circuit diameter conjecture.

\begin{theorem}\label{thm:intro-main2}
The circuit diameter of a polyhedron of the form \eqref{sys:polytope} with $\mA\in\R^{m\times n}$ is $O(m^2 \log m)$.
\end{theorem}

This result represents a significant departure from all prior work on circuit diameter. Previous upper bounds fall into two categories: (i) bounds derived by analyzing the iteration complexity of specific circuit augmentation algorithms, which inherit dependencies on numerical properties of the input, and (ii) bounds obtained by designing circuit augmentation schemes that imitate known algorithms, yielding circuit diameter bounds that approximately match the iteration counts of those algorithms. Both approaches inherit the limitations of these algorithms: they are either strongly polynomial only on highly structured polytopes, depend on encoding size, or depend on condition numbers of the constraint matrix.

In contrast, our bound of $O(m^2 \log m)$ depends \emph{only} on $m$, the number of constraints. This is the first unconditional geometric result on circuit diameter: it reveals an intrinsic property of the polyhedron's combinatorial structure, independent of how the constraint matrix is represented numerically. Our result conclusively establishes that short circuit walks exist between any pair of vertices, with length governed solely by the dimension of the ambient space.

\paragraph{Strongly polynomial versus weakly polynomial complexity.} The distinction between strongly and weakly polynomial bounds is fundamental in optimization. A bound is \emph{weakly polynomial} if it depends on the bit-complexity of the input (e.g., logarithms of numerical values or condition numbers), and \emph{strongly polynomial} if it depends only on combinatorial parameters such as the dimensions of the constraint matrix. Strongly polynomial bounds reveal that a problem's complexity is governed by its combinatorial structure rather than numerical artifacts of representation.

In linear programming, this distinction is central. Both Khachiyan's ellipsoid method \cite{Khachiyan79} and interior point methods \cite{Karmarkar84} only run in weakly polynomial time. Finding a strongly polynomial algorithm for LP remains a major open problem, listed by Smale \cite{Smale98} as one of the key mathematical challenges for the 21st century (Smale's 9th problem). Our result on circuit diameter is, to our knowledge, the first strongly polynomial bound on any diameter measure for general polyhedra.

\paragraph{Existence versus Computation.} Our construction is algorithmic: given two vertices, a short circuit walk between them can be computed in strongly polynomial time. However, this does not yield a strongly polynomial algorithm for linear programming, since computing an optimal vertex is itself the problem we wish to solve. The challenge is to find short circuit walks toward an optimum without knowing the target vertex in advance. Our result shows that no geometric obstruction prevents such walks from existing; the problem of finding short circuit walks is purely algorithmic.

Beyond settling the polynomial circuit diameter conjecture, our result suggests that polyhedra possess fundamentally good geometric structure. The focus shifts toward understanding what structure might enable efficient navigation without knowing the target: perhaps randomized or approximate methods that find near-optimal walks with high probability, or augmentation rules guided by the geometric insights underlying our proof.

\subsection{Related Work on Circuit Diameter}

We now survey related work on circuit diameter bounds, organizing by methodology.

\paragraph{Circuit augmentation algorithms and condition number bounds.} The \emph{circuit imbalance} $\kappa(\mA)$ (defined in \Cref{sec:prelim}) measures the maximum ratio between entries of an elementary vector. This parameter equals $1$ if and only if $\mA$ admits a totally unimodular representation \cite{Camion1965,Ekbatani2021}. Ekbatani et al.\ \cite{Ekbatani2021} showed that a natural extension of the Goldberg--Tarjan minimum-mean cycle canceling algorithm \cite{Goldberg89} yields a steepest-descent circuit augmentation algorithm with iteration complexity $O(n^2 m\kappa(\mA) \log(\kappa(\mA) + n))$; see also Gauthier and Desrosiers \cite{Gauthier2021}.

Dadush et al.\ \cite{Dadush2024} improved this to $O(m^2 \log(m+\kappa(\mA)))$, giving the first bound with only logarithmic dependence on $\kappa(\mA)$. Their key innovation is a ``shoot towards the optimum'' scheme: rather than greedily improving the objective, the algorithm moves directly toward a known optimal vertex $x^*$, using circuit directions that make geometric progress toward $x^*$. This more global strategy reduces dependence on local condition numbers, but still requires logarithmic dependence on the circuit imbalance.

\paragraph{Interior point methods and straight-line complexity.} Allamigeon et al.\ \cite{Allamigeon2025} introduced \emph{straight-line complexity} (SLC), a geometric measure that quantifies the curvature of trajectories in logarithmic coordinates. They show that this measure provides, up to polynomial factors in $n$, an upper bound on the number of iterations their path-following interior-point methods require. In very recent work, Dadush, Kober, and Koh \cite{dadush2026circuitdiameterstraightline} further established that SLC also bounds circuit augmentation complexity: by tracing the central path and decomposing it into ``polarized segments'', they obtain bounds of the form $O(n^2 \sum_{i\in[n]} \operatorname{SLC}(x_i^\mathfrak{m}))$.

While SLC is a natural geometric quantity, it can be exponential in the problem dimension \cite{allamigeon2018}, so this bound is not strongly polynomial. 

\paragraph{Hardness of approximation.} Black, N\"obel, and Steiner \cite{black2025shortcircuitwalksfixed} showed that approximating the shortest monotone circuit walk between two vertices (a variant where steps must not increase the objective) within a factor of $O(m^{1-\varepsilon})$ is NP-hard for any $\varepsilon>0$. This shows that efficiently computing even approximately optimal circuit walks is computationally intractable.

\paragraph{Special polytopes.} Constant and linear circuit diameter bounds have been established for specific polytopes. Borgwardt, Finhold, and Hemmecke \cite{Borgwardt2015} proved such bounds for dual transportation polyhedra. Kafer, Pashkovich, and Sanit{\`a}~\cite{Kafer2019} extended these results to matching polytopes, traveling salesman polytopes, and fractional stable set polytopes, using problem-specific structure. Borgwardt, De Loera, and Finhold~\cite{Borgwardt2016-hierarchy} introduced several variants of circuit diameter and explored their relationships.

\paragraph{Comparison of bounds.} The following table summarizes key circuit diameter bounds for general polyhedra, highlighting their dependencies:

\begin{center}
\begin{tabular}{lcc}
\toprule
\textbf{Reference} & \textbf{Bound} & \textbf{Depends on} \\
\midrule
De Loera et al.\ \cite{DKS2022} & $\operatorname{poly}(m,n,\operatorname{size}(\mA,b,c))$ & $\operatorname{size}(\mA,b,c)$ \\
Ekbatani et al.\ \cite{Ekbatani2021} & $O(n^2 m\kappa(\mA) \log(\kappa(\mA)+n))$ & $\kappa(\mA)$ \\
Dadush et al.\ \cite{Dadush2024} & $O(m^2 \log(\kappa(\mA) + m))$ & $\kappa(\mA)$ \\
Dadush et al.\ \cite{dadush2026circuitdiameterstraightline} & $O(n^2 \sum_i \operatorname{SLC}(x_i^\mathfrak{m}))$ & SLC \\
\textbf{This paper} & $\bm{O(m^2 \log m)}$ & \textbf{-} \\
\bottomrule
\end{tabular}
\end{center}
Note that $\operatorname{SLC}$ is upper bounded by $\mathrm{poly}(n)\log(n + \kappa(\mA))$ as shown in \cite{DadushKNOV24}.

Our work is the first to eliminate all dependence on numerical properties, achieving a purely combinatorial bound.

\subsection{Paper Organization}

The remainder of this paper is organized as follows. \Cref{sec:prelim} establishes notation and reviews necessary background on elementary vectors, circuits, and conformal decompositions. \Cref{sec:bound} contains our main result.

\section{Preliminaries} \label{sec:prelim}
Let $[n]=\{1,2,\ldots,n\}$.
For $\alpha\in\R$, we denote $\alpha^+=\max\{0,\alpha\}$ and $\alpha^-=\max\{0,-\alpha\}$. 
For a vector $z \in \R^n$, we define $z^+,z^-\in\R^n$ as $(z^+)_i=(z_i)^+$, $(z^-)_i=(z_i)^-$ for $i\in [n]$. 
For $z\in\R^n$, we let $\supp(z)=\{i\in [n]: z_i\neq 0\}$ denote its support, and $1/z\in(\R\cup\{\infty\})^n$ denote the vector $(1/z_i)_{i\in [n]}$. Throughout, we use the convention that $0/0 = 0$ for convenience.
We use $\|\cdot\|_p$ to denote the $\ell_p$-norm. We denote by $\R^n_+$ the non-negative orthant $\{x \in \R^n : x \ge \0\}$.

For technical reasons we assume that $m \ge 2$. For $m = 1$, our main statements are all trivial.
Note that every circuit has size at most $m+1$ since $\rk(\mA) \le m$.
The \emph{circuit imbalance measure} of $\mA$ is defined as
\[\kappa(\mA) \coloneqq \max_{g\in \EE(\mA)} \set{\frac{|g_i|}{|g_j|}:i,j\in \supp(g)}\, .\]
For $P$ as in \eqref{sys:polytope}, $x \in P$ and an elementary vector $g \in \EE(\mA)\setminus \R^n_+$, we let $\aug_P(x, g) \coloneqq x + \alpha g$ where $\alpha = \max\{\bar \alpha : x + \bar \alpha g \in P\}$. 

\begin{definition}[{\cite{bookDeLoera}}]\label{def:conformal}
  We say that $x,y\in\R^n$ are \emph{sign-compatible} if $x_iy_i\geq 0$ for all $i\in[n]$. %
  We write $x\sqsubseteq y$ if they are sign-compatible and further $|x_i|\le|y_i|$ for all $i\in [n]$.  
  For $x\in \ker(\mA)$, a \emph{conformal circuit decomposition} of $x$ is a set of elementary vectors $h^{(1)},h^{(2)},\dots,h^{(k)}$ in $\ker(\mA)$ such that $x=\sum_{j=1}^k h^{(j)}$, $k\le n-m$, and $h^{(j)}\sqsubseteq x$ for all $j\in [k]$.
\end{definition}

The following lemma shows that every vector in a linear space has a conformal circuit decomposition. It is a simple corollary of the Minkowski--Weyl and Carath\'eodory theorems.
\begin{lemma}[\cite{Dadush2024}]\label{lem:conformal}
For a matrix $\mA \in \R^{m \times n}$, every $x\in \ker(\mA)$ has a conformal circuit decomposition $x = \sum_{j=1}^k h^{(j)}$ such that $k \le \min\{\dim(\ker(\mA)), |\supp(x)|\}$.
\end{lemma}

\begin{lemma}[\cite{DadushNV20}]
A conformal circuit decomposition for a vector $x \in \ker(\mA)$ can be found in strongly polynomial time.
\end{lemma} 

\section{A strongly polynomial bound}
\label{sec:bound}

In this section we prove our main result, a strongly polynomial bound on the circuit diameter of polyhedra defined by a system of linear equations and non-negativity constraints.

\begin{theorem}
\label{thm:main-theorem}
Let $P = \{x : \mA x = b, x \ge \0\}$ be a polyhedron with $\mA \in \R^{m \times n}$. Then, for any feasible $x \in P$ and any vertex $v \in P$ there exists a circuit walk from $x$ to $v$ of length $n + O(m^2 \log(m))$. Furthermore, such a circuit walk can be constructed in strongly polynomial time.
\end{theorem}

The following corollaries follow directly from \Cref{thm:main-theorem}.

\begin{corollary}
\label{cor:main-theorem-vertex-to-vertex}
Let $P = \{x : \mA x = b, x \ge \0\}$ be a polyhedron with $\mA \in \R^{m \times n}$. Then, its circuit diameter is upper bounded by $O(m^2 \log(m))$.
\end{corollary}
\begin{proof}
Any two vertices $u,v \in P$ correspond to basic feasible solutions of \eqref{sys:polytope} and hence have support of at most $m$ variables each. Reducing the system $\mA$ to the columns indexed by $\supp(u) \cup \supp(v)$, we get a polyhedron in at most $2m$ variables. By \Cref{thm:main-theorem}, there exists a circuit walk from $u$ to $v$ of length at most $2m + O(m^2 \log(m)) = O(m^2 \log(m))$, as desired.
\end{proof}

\begin{corollary}
\label{corollary:computing-circuit-diameter}
Let $P = \{x : \mA x = b, x \ge \0\}$ be a polyhedron with $\mA \in \R^{m \times n}$. Then, its circuit diameter can be approximated up to a factor $O(m \sqrt{\log(m)})$ in strongly polynomial time. Furthermore, a circuit walk of length $O(m^2 \log(m))$ between any two vertices can be constructed in strongly polynomial time.
\end{corollary}
\begin{proof}
The second part of the corollary is trivial from \Cref{cor:main-theorem-vertex-to-vertex}. The first part is achieved by guessing a circuit diameter of $O(m\sqrt{\log(m)})$.
\end{proof}

Another important corollary concerns the running time to find a short circuit walk to an optimal vertex of a linear program. As our algorithm can find a circuit walk of length $O(m^2 \log(m))$ between any two given vertices, we can use any algorithm to solve \eqref{sys:polytope} optimally for some objective vector $c \in \R^n$ to find an optimal solution and then apply our algorithm.

\begin{corollary}
\label{corollary:slc-time}
Given a polyhedron $P = \{x : \mA x = b, x \ge \0\}$ with $\mA \in \R^{m \times n}$ and an objective $c \in \R^n$. If $\mathcal T(P)$ is the time to solve $\eqref{sys:polytope}$ optimally with objective $c$, then in time $\mathcal T(P) + \mathrm{poly}(m,n)$ we can find a circuit walk from any given initial feasible solution to an optimal vertex of $P$ of length $O(m^2 \log(m))$.
\end{corollary}

\subsection{Algorithmic Overview}

We now outline the main ideas underlying our proof, deferring technical details to \Cref{sec:formal-algorithm}. Our proof constructs an explicit circuit walk from any feasible point to a target vertex $x^*$ (with basis $B$ and non-basic indices $N = [n]\setminus B$) in $n + O(m^2\log m)$ steps. The algorithm operates in two phases.

\paragraph{Phase 1: Support reduction.} The first phase is straightforward: we reduce the number of nonzero non-basic coordinates. Since any $m+1$ variables contain a circuit of $\ker(\mA)$, whenever $|\supp(x_N)| \ge m+1$ we can find an elementary vector $g$ supported entirely within $\supp(x_N)$ and augment along it, zeroing at least one coordinate. After at most $n-2m$ steps, we achieve $|\supp(x_N)| \le m$, giving at most $2m$ nonzero coordinates total. This phase requires no sophisticated analysis and contributes only the additive $n$ term to our bound.

\paragraph{Phase 2: Trapped variables and amortized progress.} The main contribution lies in Phase 2, where we achieve the $O(m^2 \log m)$ bound through a novel combination of \emph{trapped variables} and \emph{elimination steps}.

After Phase 1, we maintain a \emph{trapped set} $T \subseteq B$ of basic indices satisfying the invariant
\[
x_t \le m \cdot x_t^* \quad \text{for all } t \in T.
\]
This invariant says that trapped coordinates are within a factor of $m$ of their target values. Crucially, once an index becomes trapped, it remains trapped throughout the algorithm. Progress in Phase 2 is measured by two events: either the trapped set $T$ grows, or a non-basic coordinate is zeroed. Since $|T| \le m$ and $|\supp(x_N)| \le m$ after Phase 1, at most $2m$ such progress events can occur.

Between progress events, the algorithm performs two types of steps:

\paragraph{Norm-reduction steps.} These steps are standard and have, e.g., been used already in \cite{Dadush2024} to make geometric progress towards a target vertex: We decompose the direction $x^* - x$ into a conformal sum of elementary vectors $g^{(1)},\ldots, g^{(k)}$ with $k \le m$ (by \Cref{lem:conformal}, the number of circuits in a conformal decomposition can be upper bounded by the kernel dimension, which is at most $m$ after Phase 1). We then select the circuit $g^{(j^*)}$ that maximizes progress on a weighted $\ell_1$-norm over $N$, and augment along it.

By conformality, each $g^{(j)}_N \le \0$ (since $x^*_N = \0$ and $x_N \ge \0$), so these steps always decrease non-basic coordinates. Crucially, the step size $\alpha$ is at most $m$, which will follow from the greedy circuit selection. This bounded step size will preserve the trapped invariant: for any $t \in T$, conformality ensures $|g_t| \le |x^*_t - x_t|$, so the change in $x_t$ is at most $m \cdot |x^*_t - x_t| \le m \cdot x^*_t$, keeping $x_t \le m x^*_t$.

Each norm-reduction step decreases the weighted $\ell_1$-norm by a $(1-1/m)$ factor. After $O(m\log m)$ consecutive norm-reduction steps, the non-basic coordinates shrink below a threshold $\tau = 1/\mathrm{poly}(m)$ relative to a reference point $x^{(\refidx)}$, triggering an elimination step.

\paragraph{Elimination steps.} When $\|x_N/x_N^{(\refidx)}\|_\infty \le \tau$, we perform an elimination step designed to force progress. The idea is to extrapolate beyond the current point in the direction it has been moving. Define
\[
y = x + \frac{\rho}{1-\rho}(x - x^{(\refidx)})\, ,
\]
where $\rho = \max_{j\in N} x_j/x_j^{(\refidx)} \le \tau$ measures how much non-basic coordinates have shrunk. By construction, $y_N \le \0$: we have extrapolated past zero on all non-basic coordinates.

However, moving directly towards $y$ via a circuit direction in the decomposition of $y - x$ might zero a trapped coordinate before any non-basic coordinate, hindering the progress we seek. To prevent this, we pull the target slightly from $y$ toward $x^*$ by setting
\[
z = y + \lambda(x^* - y)
\]
with $\lambda = \Theta(1/\poly(m))$. This perturbation is carefully calibrated: for trapped coordinates $t$ with $x_t \ll x_t^*$, it ensures $z_t \ge x_t$, preventing conformal steps toward $z$ from decreasing such coordinates. Conversely, for trapped coordinates $t$ with $x_t$ close to $x_t^*$, the construction ensures that the relative change from $x$ toward $z$ for any coordinate in $N$ dominates that of $t$.

We then decompose $z - x$ conformally and augment along the circuit that makes most progress toward zeroing some coordinate in $N$. Crucially, due to the relative change on coordinates $N$ dominating the relative change on coordinates in $T$, we will observe that either some non-basic or non-trapped basic coordinate must hit zero first, guaranteeing progress: either $|T|$ increases or $|\supp(x_N)|$ decreases.

\paragraph{A strongly polynomial bound.} The crucial property of our augmentation scheme is that our parameter $\thresh$ that we require to ensure combinatorial progress is only polynomially small and, unlike prior works, does not depend on condition numbers.
Previous approaches could not achieve this because they analyzed specific algorithmic paths guided by numerical properties (objective gradients, circuit imbalances, or curvature measures). Our algorithm, by contrast, is designed so that \emph{any} conformal decomposition suffices: we select circuits greedily for concreteness, but the analysis only requires that the decomposition exists and has size at most $m$. This allows us to bound the walk length using purely combinatorial arguments.

While \emph{trapped variables} appeared already in \cite{Dadush2024}, the main novelty of our work is the introduction of \emph{elimination steps} and the use of carefully chosen auxiliary points which force combinatorial progress after only $O(m \log m)$ norm-reduction steps, eliminating all dependence on condition numbers.

\subsection{Formal Algorithm and Proof of \Cref{thm:main-theorem}}
\label{sec:formal-algorithm}

In this section we present the full algorithm and prove \Cref{thm:main-theorem}.
The full algorithm is given in \Cref{alg:circuit-augmentation}.

\begin{algorithm}[H]
\caption{Circuit Augmentation Algorithm for Strongly Polynomial Diameter Bound}
\label{alg:circuit-augmentation}
\KwIn{Feasible solution $x^{(0)}$ to \eqref{sys:polytope} with $\mA \in \R^{m \times n}$, target vertex $x^*$ with basis $B$, threshold value $\thresh = (2m)^{-3}$, $\lambda = (2m)^{-2}$.}
\KwOut{Sequence of circuit augmentations from $x^{(0)}$ to $x^*$}

$N \gets [n] \setminus B$ \tcp*{Non-basic variables}
$T^{(-1)} \gets \emptyset$ \tcp*{Trapped variables}

\tcp{Phase 1: Reduce support on non-basic variables}
$i \gets 0$\;
\While{$|\supp(x^{(i)}_N)| \ge m + 1$ \label{line:support-reduction-begin}}{
    Select elementary vector $g \in \EE(\mA)$ with $\supp(g) \subseteq \supp(x_N^{(i)})$ and $\supp(g^-) \neq \emptyset$\;
    $x^{(i+ 1)} \gets \aug_{P}(x^{(i)}, g)$\;
    $i \gets i + 1$\;
}
\label{line:support-reduction-end}

\tcp{Phase 2: Main circuit augmentation loop}
$\refidx \gets i$ \tcp*{Reference index (reset when $T$ grows)}
\While{$x^{(i)} \neq x^*$}{
    \tcp{Update trapped set}
    $T^{(i)} \gets \{j \in B : x_j^{(i)} \le mx_j^* \}$\;
    \If{$T^{(i)} \neq T^{(i-1)}$}{
        $\refidx \gets i$ \tcp*{Reset reference point}
    }

    \tcp{Norm-reduction step: shrink non-basic coordinates}
    \If{$\|x^{(i)}_N/x^{(\refidx)}_N\|_\infty > \thresh$}{
        \label{line:norm-reduction-on-N-begin}
        $(g^{(1)}, \ldots, g^{(k)}) \gets$ conformal decomposition of $x^* - x^{(i)}$\;
        $j^* \gets \argmax_{j \in [k]} \|\frac{g^{(j)}_N}{x_N^{(\refidx)}}\|_1 $ \tcp*{Best progress on weighted $\ell_1$-norm}
        $x^{(i+1)} \gets \aug_{P}(x^{(i)}, g^{(j^*)})$ \label{line:norm-reduction-augmentation}\;
        \label{line:norm-reduction-on-N-end}
    }
    \Else{\tcp{Elimination step: force a desired coordinate to zero}
        $q \gets \argmax_{j \in N} \frac{x^{(i)}_j}{x^{(\refidx)}_j}$\;
        $\rho \gets \frac{x^{(i)}_{q}}{x^{(\refidx)}_{q}}$ \tcp*{Shrinkage ratio; note $\rho \le \thresh$}
        $y \gets x^{(i)} + \frac{\rho}{1 - \rho} (x^{(i)} - x^{(\refidx)})$ \tcp*{Extrapolate past zero on $N$}
        $z \gets y + \lambda (x^* - y)$ \tcp*{Pull toward $x^*$ to protect small trapped coords}
        $(g^{(1)}, \ldots, g^{(k)}) \gets$ conformal decomposition of $z - x^{(i)}$ \label{line:elimination-step-conformal-decomposition} \;
        $j^* \gets \argmax_{j \in [k]} -g^{(j)}_q$ \tcp*{Circuit with most progress on $q$} \label{line:elimination-circuit-selection} 
        $x^{(i+1)} \gets \aug_{P}(x^{(i)}, g^{(j^*)})$ \label{line:zero-out-variable-in-N-augmentation}\;
    }
    
    $i \gets i + 1$\;
}

\Return{Circuit walk $(x^{(0)}, x^{(1)}, \ldots, x^{(i)})$}
\end{algorithm}

\begin{proof}[Proof of \Cref{thm:main-theorem}]
We prove \Cref{thm:main-theorem} by analyzing \Cref{alg:circuit-augmentation}. 
We begin with Phase 1, which reduces the support on the non‑basic variables, and then analyze Phase 2.

\paragraph{Phase 1: Support Reduction (\LineRefRange{line:support-reduction-begin}{line:support-reduction-end}).}
Initially, if $|\supp(x^{(i)}_N)| \ge m + 1$, we perform circuit augmentations to reduce the support. Since any set of $m + 1$ variables in $\ker(\mA)$ contains a circuit, we can always find an elementary vector $g \in \EE(\mA)$ with $\supp(g) \subseteq \supp(x^{(i)}_N)$ and $\supp(g^-) \neq \emptyset$. Each augmentation reduces $|\supp(x^{(i)}_N)|$ by at least one. This phase terminates after at most $n - 2m$ iterations with $|\supp(x^{(i)}_N)| \le m$, giving us at most $2m$ non-zero variables in total.

\paragraph{Reducing the norm on variables in $N$ (\LineRefRange{line:norm-reduction-on-N-begin}{line:norm-reduction-on-N-end}).}
We maintain a set $T \subseteq B$ of \emph{trapped} variables satisfying $x_t^{(i)} \le m x_t^*$ for all $t \in T$. 

Let us first show that the set of trapped variables will only increase in the norm-reduction steps in \LineRef{line:norm-reduction-augmentation} and that furthermore, these steps make significant progress in reducing the weighted $\ell_1$ norm on $N$.
\Cref{claim:progress-on-N-norm} is analogous to progress guarantees obtained in prior work on circuit augmentation algorithms~\cite{Dadush2024}.

\begin{claim}
\label{claim:progress-on-N-norm}
If the augmentation in iteration $i$ happens in \LineRef{line:norm-reduction-augmentation}, then $T^{(i+1)} \supseteq T^{(i)}$ and furthermore, if the current reference index is $\refidx$, then
\[
    \left\|\frac{x_N^{(i+1)}}{x_N^{(\refidx)}}\right\|_1 \le \left(1 - \frac{1}{m}\right) \left\|\frac{x_N^{(i)}}{x_N^{(\refidx)}}\right\|_1\, .
\]

\end{claim}

\begin{claimproof}
Let $t \in T^{(i)}$ be any trapped variable at iteration $i$.
Let $(g^{(1)}, \ldots, g^{(k)})$ be the conformal decomposition of $x^* - x^{(i)}$, and let $j^* \in \{1, \ldots, k\}$ with $g^{(j^*)}$ being the selected circuit for augmentation in \LineRef{line:norm-reduction-augmentation}. By the induction hypothesis, we have $x_t^{(i)} \le m x_t^*$. By definition $x^{(i+1)} = x^{(i)} + \alpha g^{(j^*)}$ for some $\alpha > 0$. First, note that $g^{(j^*)}_N \le \0_N$ by conformality and $\alpha \le m$ as we have that
\begin{align}
    \label{eq:ell_1-bound}
    \begin{aligned}
    \left\|\frac{x^{(i+1)}_N}{x^{(\refidx)}_N}\right\|_1 &= \left\|\frac{x^{(i)}_N}{x^{(\refidx)}_N} + \alpha \frac{g^{(j^*)}_N}{x^{(\refidx)}_N}\right\|_1 = \left\|\frac{x^{(i)}_N}{x^{(\refidx)}_N}\right\|_1 - \alpha \left\|\frac{g^{(j^*)}_N}{x^{(\refidx)}_N}\right\|_1 \\
    &\le \left\|\frac{x^{(i)}_N}{x^{(\refidx)}_N}\right\|_1  - \frac{\alpha}{k} \left\|\frac{x^{(i)}_N}{x^{(\refidx)}_N}\right\|_1
    = \left(1- \frac{\alpha}{k}\right) \left\|\frac{x^{(i)}_N}{x^{(\refidx)}_N}\right\|_1\, , 
    \end{aligned}
\end{align} 
where the last inequality follows by the choice of $g^{(j^*)}$ maximizing the weighted $\ell_1$ norm reduction on $N$. The second equality used the fact that $g^{(j^*)}_N \le \0_N$. Using the further fact that $k \le \dim(\ker(\mA_{\supp(x^{(i)})})) \le 2m - \rk(\mA_B) = 2m - m = m$ by \Cref{lem:conformal} and the fact that the left hand side of \eqref{eq:ell_1-bound} is nonnegative gives $\alpha \le m$.
From here, note that we have for any $t \in T$ with $x_t^{(i)} \ge x_t^*$ that $g^{(j^*)}_t \le 0$ by conformality. Hence $x_t^{(i+1)} \le x_t^{(i)} \le m x_t^*$. For $t \in T$ with $x_t^{(i)} \le x_t^*$, we have by conformality that $0 \le g^{(j^*)}_t \le x_t^* - x_t^{(i)}$, giving
\[
    x_t^{(i+1)} = x_t^{(i)} + \alpha g^{(j^*)}_t \le x_t^{(i)} + \alpha (x_t^* - x_t^{(i)}) \le  x_t^{(i)} + m (x_t^* - x_t^{(i)}) \le mx_t^*\, ,
\]
which proves the first part of the claim. For the second part of the claim, note that we also have that $\alpha \ge 1$, again by conformality as we have that
\begin{equation}
\alpha = \min \left\{-\frac{x_\ell^{(i)}}{g_\ell^{(j^*)}} : \ell \in \supp([g^{(j^*)}]^{-})\right\} \ge \min \left\{\frac{x_\ell^{(i)}}{x_\ell^{(i)} - x_\ell^*} : \ell \in \supp([x^* - x^{(i)}]^-) \right\} \ge 1\, .
\end{equation} 
Therefore, we have with \eqref{eq:ell_1-bound} that
\[
    \left\|\frac{x^{(i+1)}_N}{x^{(\refidx)}_N}\right\|_1 \le \left(1 - \frac{\alpha}{k}\right) \left\|\frac{x^{(i)}_N}{x^{(\refidx)}_N}\right\|_1 \le \left(1 - \frac{1}{m}\right) \left\|\frac{x^{(i)}_N}{x^{(\refidx)}_N}\right\|_1\, ,
\]
which proves the claim.
\end{claimproof}

A direct consequence of \Cref{claim:progress-on-N-norm} is the following.

\begin{claim}
Within $m\log(m/\thresh)$ many consecutive executions of \LineRef{line:norm-reduction-augmentation} either the set of trapped variables $T$ gets extended or the if condition $\|x_N^{(i)}/x_N^{(\refidx)}\|_\infty > \thresh$ in \LineRef{line:norm-reduction-on-N-begin} becomes false.
\end{claim}
\begin{claimproof}
Let $i$ be an iterate and let $p = i + m \cdot \log(m/\thresh)$ be such that in each iteration $j \in \{i, \ldots, p-1\}$ the augmentation is performed in \LineRef{line:norm-reduction-augmentation}. If $T$ has not been extended in any of these iterations, then the reference index $\refidx \le i$ remains unchanged. By \Cref{claim:progress-on-N-norm}, we have that
     \begin{align*}
     \left\|\frac{x_N^{(p)}}{x_N^{(\refidx)}}\right\|_\infty &\le \left\|\frac{x_N^{(p)}}{x_N^{(\refidx)}}\right\|_1 \le \left(1 - \frac{1}{m}\right)^{p - i} \left\|\frac{x_N^{(i)}}{x_N^{(\refidx)}}\right\|_1 = \left(1 - \frac{1}{m}\right)^{m\log(m/\thresh)} \left\|\frac{x_N^{(i)}}{x_N^{(\refidx)}}\right\|_1 \le \left(1 - \frac{1}{m}\right)^{m\log(m/\thresh)} \cdot m 
     \\ 
     & \le e^{-\log(m/\thresh)} \cdot m\le \frac{\thresh}{m} \cdot m = \thresh\, ,
     \end{align*}
as desired.
\end{claimproof}

We now claim that a single execution of the elimination step in \LineRef{line:zero-out-variable-in-N-augmentation} either sets a variable in $N$ to zero or extends $T^{(\refidx)}$ (= $T^{(i)}$). From now on, we let $T \coloneqq T^{(\refidx)}$ for simplicity.

\paragraph{Elimination step analysis}
Let $i > \refidx$ be such that the condition $\|x_N^{(i)}/x_N^{(\refidx)}\|_\infty \le \thresh$ holds. Let $q \in N$ be the index maximizing $x_q^{(i)}/x_q^{(\refidx)}$ and let $\rho \coloneqq x_q^{(i)}/x_q^{(\refidx)} \le \thresh$. As in the algorithm, we define 
the vectors
\[
y \coloneqq x^{(i)} + \frac{x_q^{(i)}}{x_q^{(\refidx)} - x_q^{(i)}} (x^{(i)} - x^{(\refidx)}) = x^{(i)} + \frac{\rho}{1- \rho} (x^{(i)} - x^{(\refidx)})\, , \qquad z \coloneqq y + \lambda(x^* - y)\, .
\]
Note that $y_q = 0$ and $z_q = 0$. By construction, we have $y_N \le \0_N$ as for any $j \in N$, we get that 
\[
(1 - \rho) y_j = (1- \rho)x_j^{(i)} + \rho (x_j^{(i)} - x_j^{(\refidx)}) = x_j^{(i)} - \rho x_j^{(\refidx)} \le x_j^{(i)} - \frac{x_j^{(i)}}{x_j^{(\refidx)}} x_j^{(\refidx)} = 0.
\]

For all $t \in T$, we have
\begin{align*}
    z_t - x_t^{(i)} &= \lambda (x_t^* - y_t) + \frac{\rho}{1- \rho}(x_t^{(i)} - x_t^{(\refidx)})  \\
&= \lambda(x_t^* - x_t^{(i)}) + \lambda (x_t^{(i)} - y_t) +  \frac{\rho}{1- \rho}(x_t^{(i)} - x_t^{(\refidx)}) \\
&= \lambda(x_t^* - x_t^{(i)}) + \frac{(1-\lambda)\rho}{1- \rho}(x_t^{(i)} - x_t^{(\refidx)})\, .
\end{align*}

Now, by \Cref{claim:progress-on-N-norm} we have that $x_t^{(i)}, x_t^{(\refidx)} \le m x_t^*$. Therefore, we have that
\begin{equation}
\label{eq:elimination_proximity_bound}
\begin{aligned}
\frac{|z_t - x_t^{(i)}|}{x_t^{*}} &= \frac{\left|\lambda(x_t^* - x_t^{(i)}) + \frac{(1-\lambda)\rho}{1-\rho}(x_t^{(i)} - x_t^{(\refidx)})\right|}{x_t^{*}} \\
& \le \lambda \frac{|x_t^* - x_t^{(i)}|}{x_t^{*}} + \frac{(1-\lambda)\rho}{1-\rho}\frac{|x_t^{(i)} - x_t^{(\refidx)}|}{x_t^{*}} \\
&\le \left(\lambda + \frac{(1-\lambda)\rho}{1-\rho}\right) m\,. 
\end{aligned}
\end{equation}

Observe that $\lambda = (2m)^{-2} \le 1$, which implies that $z$ is a convex combination of $y$ and $x^*$. Consequently, $z_N \le \0_N$.

We now show that when we decompose $z - x^{(i)}$ conformally (\LineRef{line:elimination-step-conformal-decomposition}) and select $g^{(j^*)}$ maximizing $-g^{(j)}_q$ (\LineRef{line:elimination-circuit-selection}), the augmentation (\LineRef{line:zero-out-variable-in-N-augmentation}) will zero out a coordinate in $N$ or $B \setminus T$ before any coordinate in $T$.

The coordinate that gets zeroed out first is 
\[
    \argmax_{j : g^{(j^*)}_j < 0} \frac{-g^{(j^*)}_j}{x_j^{(i)}}\, .
\]

For $t \in T$ with $x_t^{(i)} \le \tfrac12 x_t^*$, we have that
\begin{align*}
 z_t - x_t^{(i)} &= \lambda(x_t^* - x_t^{(i)}) + \frac{(1-\lambda)\rho}{1-\rho}(x_t^{(i)} - x_t^{(\refidx)})\\
&\ge \lambda(x_t^* - x_t^{(i)}) - \frac{(1-\lambda)\rho}{1-\rho} \left|x_t^{(i)} - x_t^{(\refidx)}\right| \\
&\ge \lambda \frac{x_t^*}{2} - \frac{(1-\lambda)\rho}{1-\rho} \cdot m x_t^* \\
&= \left(\frac{\lambda}{2} - \frac{(1-\lambda)\rho}{1-\rho}m\right) x_t^* \\
&\ge \left(\frac{\lambda}{2} - \rho m\right) x_t^* \\
&\ge 0\, .
\end{align*}
By conformality, we therefore have for all $t \in T$ with $x_t^{(i)} \le \tfrac12 x_t^*$ that $g^{(j^*)}_t \ge 0$. Hence, such coordinates cannot be zeroed out in the conformal augmentation. For $t \in T$ with $x_t^{(i)} \ge \tfrac12 x_t^*$ we have that 
\begin{align}
\label{eq:bound_elimination_on_T}
\frac{-g^{(j^*)}_t} {x_t^{(i)}} &\le \frac{|z_t - x_t^{(i)}|}{x_t^{(i)}} \nonumber \\
&\le  \frac{2|z_t - x_t^{(i)}|}{x_t^{*}} \nonumber \\
&\le 2\left(\lambda + \frac{(1-\lambda)\rho}{1-\rho}\right) m \tag{by \eqref{eq:elimination_proximity_bound}}  \\
&\le 2\left(\lambda + \frac{(1-\lambda)\thresh}{1-\thresh}\right)m \tag{as $\rho \le \thresh$}  \\
&= 2\left((2m)^{-2} + \frac{(1-(2m)^{-2})\cdot (2m)^{-3}}{1-(2m)^{-3}}\right)m \nonumber \\
&< \frac{1}{m}\, . \nonumber 
\end{align}

On the other hand, we have that $z_q = 0$ and so $\sum_j g_q^{(j)} = z_q - x_q^{(i)} = -x_q^{(i)}$ and so 
\begin{equation*}
\max_{j \in N} \frac{-g_j^{(j^*)}}{x_j^{(i)}} \ge \frac{-g_q^{(j^*)}}{x_q^{(i)}} \ge \frac{1}{k} \sum_{j = 1}^k \frac{-g_q^{(j)}}{x^{(i)}_q} = \frac{1}{k} \ge \frac{1}{m} \, .
\end{equation*} 
So, the two inequalities above yield 
 \begin{equation}
    \max_{j \in N} \frac{-g_j^{(j^*)}}{x_j^{(i)}} > \max_{t \in T} \frac{-g_t^{(j^*)}}{x_t^{(i)}}\, .
 \end{equation}
 Therefore, the augmentation in \LineRef{line:zero-out-variable-in-N-augmentation} cannot set a variable in $T$ to zero. As a consequence, either a variable in $N$ is set to zero, or a variable in $B \setminus T$ is set to zero. In the former case, this variable will remain zero forever, while in the latter case, the variable that is set to zero will now satisfy $x_j^{(i+1)} = 0 \le mx_j^*$ and so $T$ gets extended.

With this we proved the main part of the theorem. 
 To show that the algorithm is correct we still have to show that after this augmentation, variables that originated in $T$ remain trapped.

\begin{claim}
\label{claim:elimination-trapped-invariant}
Let $i$ be an iterate where the augmentation is performed in \LineRef{line:zero-out-variable-in-N-augmentation} and let $x^{(i+1)}$ be the resulting point. Then, $T^{(i)} \subseteq T^{(i+1)}$.
\end{claim}
\begin{claimproof}
For any $t \in T$ with $x_t^{(i)} \ge z_t$, we have by conformality that $g^{(j^*)}_t \le 0$ and so $x_t^{(i+1)} \le x_t^{(i)} \le mx_t^*$. 

For $t \in T$ with $x_t^{(i)} \ge \tfrac32 x_t^*$, we get that
\begin{equation*}
    z_t - x_t^{(i)} = \lambda (x_t^* - x_t^{(i)}) + \frac{(1-\lambda)\rho}{1-\rho} (x_t^{(i)} - x_t^{(\refidx)}) \le -\frac{\lambda}{2} x_t^* + \frac{(1-\lambda)\rho}{1-\rho} m x_t^* \le \left(-\frac{\lambda}{2} + \rho m\right) x_t^* \le \left(-\frac{\lambda}{2} + \thresh m\right) x_t^* = 0\, ,
\end{equation*}
and so $g^{(j^*)}_t \le 0$. 

It therefore remains to consider the case $x_t^{(i)} \le \tfrac32 x_t^*$ and $z_t \ge x_t^{(i)} $.
 Now, note that similar to the proof of \Cref{claim:progress-on-N-norm}, we have that the step size $\alpha$ in the augmentation $x^{(i+1)} = x^{(i)} + \alpha g^{(j^*)}$ satisfies $\alpha \le m$.
This can be observed by noting that 
\begin{equation}
    -g_q^{(j^*)} = \max_{j \in [k]} -g_q^{(j)} \ge \frac{1}{k} \sum_{j = 1}^k -g_q^{(j)} = \frac{1}{k} (x_q^{(i)} - z_q) = \frac{1}{k}x_q^{(i)} \ge \frac{1}{m} x_q^{(i)}\, ,
\end{equation} 
where we used again that $z_q = 0$. Hence, we obtain
\begin{equation}
\alpha = \min \left\{-\frac{x_j^{(i)}}{g_j^{(j^*)}} : j \in \supp([g^{(j^*)}]^{-})\right\} \le -\frac{x_q^{(i)}}{g_q^{(j^*)}} \le m\, .
\end{equation}
 But then, we have with \eqref{eq:elimination_proximity_bound} that
\begin{align}
x_t^{(i+1)} &= x_t^{(i)} + \alpha g^{(j^*)}_t \notag \\
&\le x_t^{(i)} + \alpha (z_t - x_t^{(i)}) \tag{by conformality} \\
& \le x_t^{(i)} + m\left(\lambda + \frac{(1-\lambda)\rho}{1-\rho}\right) m x_t^* \tag{by $\alpha \le m$ and \eqref{eq:elimination_proximity_bound}}\\
& \le \frac{3}{2}\left(1 + m^2 \cdot \left(\lambda + \frac{(1-\lambda)\rho}{1-\rho}\right)\right) x_t^*  \tag{by $x_t^{(i)} \le \frac{3}{2} x_t^*$}\\
& \le \frac{3}{2}\left(1 + m^2 \cdot \left(\lambda + \rho\right)\right) x_t^*  \tag{by $\lambda > \rho$} \\
& \le \frac{3}{2}\left(1 + m^2 \cdot \left((2m)^{-2}+ (2m)^{-3}\right)\right) x_t^* \notag \\
& \le m x_t^* \notag\, , 
\end{align} 
where the last inequality follows by the assumption that $m \ge 2$.
This proves the claim.
\end{claimproof}

It remains to analyze the total number of circuit augmentations of the algorithm.
By \Cref{claim:progress-on-N-norm}, we perform at most $m \log(m/\thresh)$ consecutive norm-reduction iterations before either $T$ gets extended or we reduced the relative norm on $N$ below $\thresh$ so that the elimination step is called, which either extends $T$ or sets a variable in $N$ to zero. Since $|T| \le m$ and $|N| \le m$ after Phase 1, there are at most $2m$ such progress events, so the total number of augmentations after Phase 1 is at most $(2m) \cdot m \log(m/\thresh) = O(m^2 \log(m))$.
\end{proof}

\subsection{Monotone Diameter}

A circuit walk is \emph{monotone} with respect to an objective $c \in \R^n$ if the objective value does not increase along the walk. The \emph{monotone circuit diameter} measures the worst-case length of shortest monotone walks.

\begin{definition}
\label{def:monotone-diameter}
Let $P = \{x : \mA x = b, x \ge \0\} \subseteq \R^n$ be a polyhedron. A circuit walk $(x^{(0)}, x^{(1)}, \ldots, x^{(k)})$ is \emph{monotone with respect to $c \in \R^n$} if $c^\top x^{(i+1)} \le c^\top x^{(i)}$ for all $i \in \{0, \ldots, k-1\}$. The \emph{monotone circuit diameter} of $P$ is the maximum, over all vertices $u, v \in P$ and objectives $c \in \R^n$ with $v \in \argmin_{x \in P} c^\top x$, of the length of the shortest monotone circuit walk from $u$ to $v$.
\end{definition}

Very recently, Black, N\"obel, and Steiner~\cite{black2025shortcircuitwalksfixed} showed that approximating the shortest monotone circuit walk between two vertices is hard.

\begin{theorem}[{\cite[Corollary 1.7]{black2025shortcircuitwalksfixed}}]
\label{thm:hardness}
For every $\eps > 0$ and $m \ge 2$, the following problem is NP-hard: Given a polyhedron $P = \{x : \mA x = b, x \ge \0\} \subseteq \R^n$ with $\mA \in \R^{m \times n}$, two vertices $u, v \in P$, and an objective $c$ minimized at $v$, compute a monotone circuit walk from $u$ to $v$ approximating the minimum possible length of such a walk to within a factor of $O(m^{1 - \eps})$.
\end{theorem}

Since every circuit walk is trivially monotone for $c = \0$, the circuit diameter is a lower bound on the monotone circuit diameter. We show that \Cref{alg:circuit-augmentation} produces monotone walks, yielding a strongly polynomial bound on the monotone circuit diameter as well.

The key observation is that monotonicity can be characterized in terms of the non-basic coordinates. Let $v = x^*$ be the target vertex with basis $B$ and non-basic indices $N = [n] \setminus B$. By complementary slackness, any objective $c$ minimized at $x^*$ satisfies $c = \mA^\top y + s$ for some $y \in \R^m$ and $s \in \R^n$ with $s_B = \0$ and $s_N \ge \0$. For any circuit step $x^{(i)} \to x^{(i+1)} = x^{(i)} + \alpha g$ with $g \in \ker(\mA)$, we have
\[
c^\top(x^{(i+1)} - x^{(i)}) = \alpha \cdot c^\top g = \alpha \cdot s^\top g = \alpha \sum_{j \in N} s_j g_j\, .
\]
Since $\alpha > 0$ and $s_N \ge \0$, the step is monotone if $g_N \le \0$. Thus, an algorithm that only uses circuit directions $g$ with $g_N \le \0$ is monotone with respect to all objectives minimized at $x^*$.

\Cref{alg:circuit-augmentation} satisfies this property. In the norm-reduction steps, the conformal decomposition of $x^* - x^{(i)}$ yields circuits $g^{(j)} \sqsubseteq x^* - x^{(i)}$. Since $x^*_N = \0$ and $x^{(i)}_N \ge \0$, we have $(x^* - x^{(i)})_N \le \0$, and conformality implies $g^{(j)}_N \le \0$. In the elimination steps, the target $z$ satisfies $z_N \le \0$ by construction, and since $x^{(i)}_N \ge \0$, we again have $(z - x^{(i)})_N \le \0$, so the conformal decomposition yields circuits with $g^{(j)}_N \le \0$.

\begin{corollary}
\label{cor:monotone-diameter}
Let $P = \{x : \mA x = b, x \ge \0\} \subseteq \R^n$ be a polyhedron with $\mA \in \R^{m \times n}$. Then, the monotone circuit diameter of $P$ is bounded by $O(m^2 \log(m))$.
\end{corollary}

Combined with the hardness result in \Cref{thm:hardness}, our result leaves a gap of $O(m^{1+\eps} \log(m))$ between the best known upper and lower bounds for approximating the shortest monotone circuit walk between vertices.

\newcommand{\etalchar}[1]{$^{#1}$}

\end{document}